\documentclass[11pt]{article}
\usepackage{amssymb, amsmath, hyperref, array}

\def\R{\mathbb{R}}
\def\N{\mathbb{N}}

\newtheorem{defn}{Definition}

\newtheorem{lemma}[defn]{Lemma}
\newtheorem{proposition}[defn]{Proposition}
\newtheorem{theorem}[defn]{Theorem}

\newenvironment{proof}[1]{
  \trivlist \item[\hskip \labelsep{\it #1}]}{\hfill\mbox{$\square$}
  \endtrivlist}

\title{A few more extensions of Putinar's Positivstellensatz \\ to non-compact sets}

\date{}

\author{Paula Escorcielo\footnote
{{\scriptsize Partially supported by the Argentinian grants} {\footnotesize UBACYT 20020190100116BA} 
{\scriptsize and} {\footnotesize PIP 11220130100527CO CO\-NI\-CET}. 
} 
\qquad \qquad
Daniel Perrucci$^{*}$
\\[3mm]
{\small Departamento de Matem\'atica, FCEN, Universidad de Buenos Aires, Argentina}\\ 
{\small  IMAS, CONICET--UBA, Argentina}
}

\begin{document}

\maketitle

\begin{abstract}
We extend previous results about Putinar's Positivstellensatz for cylinders of type $S \times {\mathbb R}$
to sets of type $S \times {\mathbb R}^r$ in some special cases 
taking into account $r$ and the degree of the polynomial
with respect to the variables
moving in ${\mathbb R}^r$ (this is to say, in the non-bounded directions). 
These special cases are in correspondence with the ones where the equality between the 
cone of non-negative polynomials and the cone of sums of squares holds.  Degree bounds are provided.  
\end{abstract}

\section{Introduction}

One of the most important 
results in the theory of sums of squares and certificates of non-negativity is 
Putinar's Positivstellensatz (\cite{Put}). 
This theorem states that
given $g_1, \dots, g_s \in \R[\bar X] = \R[X_1, \dots, X_n]$ such that the quadratic module generated by $g_1, \dots, g_s$,
$$
M(g_1, \dots, g_s) = \left\{\sigma_0 + \sigma_1g_1 + \dots + \sigma_sg_s \ | \ 
\sigma_0, \sigma_1, \dots, \sigma_s \in \sum \R[\bar X]^2\right\}
\subset \R[\bar X]
$$
is archimedean,  every 
$f \in \R[\bar X]$ positive on 
$$
S = \{\bar x \in \R^n \, | \, g_1(\bar x) \ge 0, \dots, g_s(\bar x) \ge 0\}
$$
belongs to $M(g_1, \dots, g_s)$.
An explicit expression of $f$ as a member of $M(g_1, \dots, g_s)$  is a certificate of the non-negativity of
$f$ on $S$. 
Note that the condition of
archimedeanity on
$M(g_1, \dots, g_s)$  
implies that $S$ is compact 
(see for instance \cite[Chapter 5] {Marshall_book} for the definition and equivalences of archimedeanity). 

It is of interest to look for a bound for the degrees of all the different terms  
in the representation of $f$ as en element of 
$M(g_1, \dots, g_s)$ which existence is ensured by Putinar's Positivstellensatz. 
An answer to this question 
was given by  Nie and Schweighofer (\cite[Theorem 6]{NieSchw}).
In the particular case where $S$ is the hypercube $[0, 1]^n$, 
improved bounds were given in \cite{deKLau} and \cite{Mag}.

Consider the norm $\| \cdot \|$ in $\R[\bar X]$ defined as follows. 
$$
\hbox{For }\, f = \sum_{\substack{\alpha \in \N_{0}^n \\ |\alpha| \le d}} \binom{|\alpha|}{\alpha}a_\alpha \bar X^\alpha,
\quad \quad
\| f \| = \max \{ |a_\alpha| \, | \, \alpha \in \N_{0}^n, |\alpha| \le d \} 
$$ 
where for 
$\alpha = (\alpha_1, \dots, \alpha_n) \in  \N_{0}^n$, $
|\alpha| = \alpha_1 + \dots + \alpha_n$
and
$
\binom{|\alpha|}{\alpha} = \frac{|\alpha|!}{\alpha_1!\dots \alpha_n!} \in \N.
$
Then \cite[Theorem 6]{NieSchw} is the following result.  

\begin{theorem}[Putinar's Positivstellensatz with degree bound]\label{thm:NieSchw}
Let
$g_1, \dots, g_s \in \R[\bar X]$ 
such that 
$$
\emptyset \ne S  = \{\bar x \in \R^{n} \ | \ g_1(\bar x) \ge 0, \dots, g_s(\bar x) \ge 0\} \subset 
(-1, 1)^n,
$$
and the quadratic module $M(g_1, \dots, g_s)$ is archimedean. 
There exists a positive constant $c$ such that for every $f \in \R[\bar X]$  
positive on $S$, if
$\deg f = d$ and
$\min\{f(\bar x) \,  | \, \bar x \in S \} = f^* > 0$,   
then $f$ can be written as
$$
f = \sigma_0 + \sigma_1g_1 + \dots + \sigma_sg_s \in M(g_1, \dots, g_s)
$$
with $\sigma_0, \sigma_1,  \dots, \sigma_s \in \sum \R[\bar X]^2$ and 
$$
\deg(\sigma_0), \deg(\sigma_1g_1),  \dots, \deg(\sigma_sg_s)
\le 
c \, {\rm e}^{\left(\frac{\| f \|d^2n^d}{f^*}\right)^c}.
$$
\end{theorem}

A natural question is if it is possible to relax the 
archimedeanity hyphothesis 
and to 
extend Putinar's Positivstellensatz to 
cases where $S$ is non-compact. 
Even though it is well-known that this is not possible in full generality, results in this direction were given in 
\cite{KuhlMars}, \cite{KuhlMarsSchw}, \cite{Mar} and \cite{EscPer}.
We introduce the notation and definitions needed to state our main result from 
\cite{EscPer}.

We note
$$C = \{(y, z) \in \R^2 \ | \  y^2 + z^2 = 1\}.$$
For 
$$
f = \sum_{0 \le i \le m} f_i(\bar X) Y^i \in \R[\bar{X},Y]
$$
with $\deg_Y f = m$, we note 
$$
\bar f = \sum_{0 \le i \le m} f_i(\bar X) Y^iZ^{m-i} \in \R[\bar{X},Y, Z]
$$
its homogeneization only with respect to the variable $Y$.
Such an $f$ is said to be \emph{fully} $m$-\emph{ic} on $S$
if for every $\bar x \in S$, $f_m(\bar x) \ne 0$.
We define the norm $\| \cdot \|_{\bullet}$ on $\R[\bar X, Y]$ as follows.  
$$
\hbox{For }\, 
 f= \sum_{0 \le i \le m} \, \sum_{\substack{\alpha \in \N_{0}^n\\|\alpha| \le d}} \binom{|\alpha|}{\alpha} a_{\alpha,i} \bar{X}^\alpha Y^i, \quad \quad  
 \| f \|_{\bullet}= \max\{|a_{\alpha,i}| \, | \, 0 \le  i \le m, \alpha \in \N_0^n, |\alpha| \le d\}.
$$
Note that in this norm, the variable $Y$ 
is not considered in the same way than the variables $\bar X$. 
Finally, for $g_1, \dots, g_s \in \R[\bar X]$, we note
$$
M_{\R[\bar X, Y]}(g_1, \dots, g_s) = \left\{ \sigma_0 + \sigma_1g_1 + \dots + \sigma_sg_s \, | \, 
\sigma_0, \sigma_1, \dots, \sigma_s \in \sum \R[\bar X, Y]^2\right\} \subset \R[\bar X, Y]$$
the quadratic module generated by $g_1, \dots, g_s$ in $\R[\bar X, Y]$,  while
the notation $M(g_1, \dots, g_s)$ is kept for 
the quadratic module generated by $g_1, \dots, g_s$ in $\R[\bar X]$.

In \cite[Theorem 7]{EscPer}, we prove an 
extension of Putinar's
Positivstellensatz to cylinders of type $S \times \R$
which is the following result.
\begin{theorem}\label{th:main_paper_anterior}
Let
$g_1, \dots, g_s \in \R[\bar X]$ 
such that 
$$
\emptyset \ne S  = \{\bar x \in \R^{n} \ | \ g_1(\bar x) \ge 0, \dots, g_s(\bar x) \ge 0\} \subset (-1, 1)^n
$$
and the quadratic module $M(g_1, \dots, g_s) \subset \R[\bar X]$ is archimedean. 
There exists a positive constant $c$ such that for every $f \in \R[\bar X, Y]$ positive on $S \times \R$, 
if $\deg_{\bar X} f = d$, $\deg_{Y} f = m$ with $f$ fully $m$-ic on $S$ and 
$$
\min\{\bar f(\bar x, y, z) \,  | \, \bar x \in S, (y, z) \in C \} = f^{\bullet} >0, 
$$ 
then $f$ can be written as
$$
f = \sigma_0 + \sigma_1g_1 + \dots + \sigma_sg_s \in M_{\R[\bar X, Y]}(g_1, \dots, g_s)
$$
with $\sigma_0, \sigma_1,  \dots, \sigma_s \in \sum \R[\bar X, Y]^2$ and 
$$
\deg(\sigma_0), \deg(\sigma_1g_1),  \dots, \deg(\sigma_sg_s) 
\le c (m+1)2^{\frac{m}{2}}  {\rm e}^{ \left( \frac{\| f \|_{\bullet}(m+1)d^2(3n)^d}{f^{\bullet}} \right)^{c} }.
$$
\end{theorem}

The condition of $f$ being fully $m$-ic was defined originally in \cite{Pow}.
Under this assumption, in  
\cite[Theorem 3]{Pow} Powers obtains an extension of
Schm\"udgen's 
Positivstellensatz (\cite{Schm}) to cylinders of
type $S \times F$ with $S \subset \R^n$ a 
compact semialgebraic set and $F \subset \R$ an unbounded closed semialgebraic set. 
Indeed, the general idea to prove \cite[Theorem 7]{EscPer} is 
the same as in 
\cite[Theorem 3]{Pow}, which is to 
consider the variable $Y$ 
as a parameter
and to obtain for each specialization of $Y = y \in F$ a certificate of the non-negativity of $f(\bar X, y)$ on $S$, in a uniform way such that all these certificates
can be glued together to obtain the desired representation for $f(\bar X, Y)$. 

Both in \cite[Theorem 3]{Pow} (in the case $F = \R$) and in 
\cite[Theorem 7]{EscPer}, in the gluing process, it is used that every univariate polynomial non-negative on 
$\R$ is a sum of squares in $\R[Y]$.
This suggests that the same approach can also be used in every 
other possible situation where the equality between the 
cone of non-negative polynomials and the cone of sums of squares holds, 
which are known to be the case of 
bivariate polynomials of degree $4$ 
and the case of 
multivariate polynomials of degree $2$ 
(see for instance 
\cite[Chapter 6]{BCR}). 
Indeed, if separate sets of variables are considered, then the equality 
also holds in the particular case of polynomials in two sets of variables, 
the first set consisting in a single variable, and the degree of the polynomial with respect to the second set of variables equal to $2$
(see for instance \cite{Djo} or \cite[Section 7]{CLR}).

We extend the notation introduced before. 
For $r \in \N$,   we
note $\bar Y = (Y_1, \dots, Y_r)$ and
$$
C^r = \{(\bar y, z) \in \R^{r+1} \ | \ y_1^2 + \dots +
y_r^2 + z^2 = 1\}. 
$$

Also, for even $m \in \N$ and 
$$
f= \sum_{\substack{\beta \in \N_0^r \\|\beta|\le m }}  f_{\beta}(\bar X) \bar Y^{\beta} \in \R[\bar{X},\bar Y]
$$
with $\deg_{\bar Y} f= m$, we note 
$$
\bar f = \sum_{\substack{\beta \in \N_0^r \\|\beta|\le m }}  f_{\beta}(\bar X) \bar Y^{\beta}Z^{m - |\beta|} \in \R[\bar{X},\bar Y, Z]
$$ 
its homogeneization only with respect to the variables $\bar Y$.
For such an $f$, we say that $f$ satisfies the condition $(\dag)$ on $S$ if 
for every $\bar x \in S$, 
$$
\sum_{\substack{\beta \in \N_0^r \\|\beta| = m }}  f_{\beta}(\bar x) \bar Y^{\beta}
\hbox{ is a positive definite } m\hbox{-form in } \R^{r}.
$$
Note that if $f$ is positive in $S \times \R^r$ and satisfies condition $(\dag)$ on $S$, 
then $\bar f$ is positive on $S \times C^r$.

We extend the definition of the norm $\| \cdot \|_{\bullet}$ 
to $\R[\bar X, \bar Y]$ as follows.  
$$
\hbox{For }\, f= \sum_{\substack{\beta \in \N_{0}^r\\|\beta| \le m}} \, \sum_{\substack{\alpha \in \N_{0}^n\\|\alpha| \le d}} 
\binom{|\alpha|}{\alpha} a_{\alpha,\beta} \bar{X}^\alpha {\bar Y}^{\beta}, 
\quad \quad 
\| f \|_{\bullet}= \max\{|a_{\alpha,\beta}| \, | \,  
\beta \in \N_0^r, 
|\beta| \le m,
\alpha \in \N_0^n, 
|\alpha| \le d\}.
 $$
Finally, for $g_1, \dots, g_s \in \R[\bar X]$, we note
$$
M_{\R[\bar X, \bar Y]}(g_1, \dots, g_s) = \left\{ \sigma_0 + \sigma_1g_1 + \dots + \sigma_sg_s \, | \, 
\sigma_0, \sigma_1, \dots, \sigma_s \in \sum \R[\bar X, \bar Y]^2\right\} \subset \R[\bar X, \bar Y]$$
the quadratic module generated by $g_1, \dots, g_s$ in $\R[\bar X, \bar Y]$. 

Both the norm $\| \cdot \|_\bullet$ and the quadratic module $M_{\R[\bar X, \bar Y]}$ will be used later on in this paper with many different 
vectors of variables $\bar Y$, but always distinguishing the 
same vector of variables $\bar X$. 

In this setting, Theorem \ref{th:main_paper_anterior} (\cite[Theorem 7]{EscPer}) is the extension of Putinar's Positivstellensatz corresponding to the 
case $r = 1$. 
We present
below the two new extensions corresponding to the cases $r = 2, m = 4$ and $m = 2$.

\begin{theorem}\label{th:2_variables_grado_4}
Let
$g_1, \dots, g_s \in \R[\bar X]$ 
such that 
$$
\emptyset \ne S  = \{\bar x \in \R^{n} \ | \ g_1(\bar x) \ge 0, \dots, g_s(\bar x) \ge 0\} \subset (-1, 1)^n
$$
and the quadratic module $M(g_1, \dots, g_s) \subset \R[\bar X]$ is archimedean.
There exists a positive constant $c$ such that for every $f \in \R[\bar X, Y_1,Y_2]$ positive on $S \times \R^2$, 
if $\deg_{\bar X} f = d$, $\deg_{ (Y_1, Y_2)}f=4$, 
$f$ satisfies condition $(\dag)$ on $S$ and
$$
\min \left\{
\bar f(\bar x, y_1, y_2, z) \, | \, \bar x \in S, (y_1, y_2, z) \in 
C^2
\right\} = f^{\bullet} > 0,
$$
then $f$ can be written as
$$
f= \sigma_0+\sigma_1g_1+\dots+\sigma_sg_s \in M_{\R[\bar X, Y_1, Y_2]}(g_1, \dots, g_s)
$$
with $\sigma_0,\sigma_1, \dots, \sigma_s \in \sum \R[\bar X, Y_1, Y_2]^2$ and
$$
\deg(\sigma_0),\deg(\sigma_1g_1), \dots,\deg( \sigma_sg_s) \leq 
c
{\rm e}^{\left( \frac{\|f \|_{\bullet}d^2(3n)^d}{f^{\bullet}} \right)^{c} }.
$$
\end{theorem}

\begin{theorem}\label{th:r_variables_grado_2}
Let
$g_1, \dots, g_s \in \R[\bar X]$ 
such that 
$$
\emptyset \ne S  = \{\bar x \in \R^{n} \ | \ g_1(\bar x) \ge 0, \dots, g_s(\bar x) \ge 0\} \subset (-1, 1)^n
$$
and the quadratic module $M(g_1, \dots, g_s) \subset \R[\bar X]$ is archimedean.
There exists a positive constant $c$ such that for every $f \in \R[\bar X, \bar Y]$ positive on $S \times \R^r$, 
if $\deg_{\bar X} f= d$, 
$\deg_{\bar Y}f=2$,
$f$ satisfies condition $(\dag)$ on $S$ 
and
$$
\min \left\{
\bar f(\bar x, \bar y, z) \, | \, \bar x \in S, (\bar y, z) \in 
C^r
\right\} = f^{\bullet} > 0,
$$
then $f$ can be written as
$$
f= \sigma_0+\sigma_1g_1+\dots+\sigma_sg_s \in M_{\R[\bar X, \bar Y]}(g_1, \dots, g_s)
$$
with $\sigma_0,\sigma_1, \dots, \sigma_s \in \sum \R[\bar X, \bar Y]^2$ and
$$
\deg(\sigma_0),\deg(\sigma_1g_1), \dots,\deg( \sigma_sg_s) \leq 
cr^2
{\rm e}^{ \left( \frac{ \| f \|_{\bullet}r^2d^2(3n)^d}{{f^{\bullet}}} \right)^{c} }.
$$
\end{theorem}

Finally, we introduce the notation we need to state the extension corresponding to the case of two separate sets of variables.
For $r \in \N$, we
note $\bar Y_2 = (Y_{21}, \dots, Y_{2r})$. Also, 
for even $m \in \N_0$ and 
$$
f= \sum_{0 \le i \le m} \sum_{\substack{\beta \in \N_0^r \\|\beta|\le 2 }}  f_{i,\beta}(\bar X) Y_{1}^i \bar Y_2^{\beta} \in \R[\bar{X},Y_{1}, \bar Y_2]
$$
with  $\deg_{Y_1} f = m$ and $\deg_{\bar Y_2} f= 2$, we note 
$$
\bar {\bar f} = \sum_{0 \le i \le m} 
\sum_{\substack{\beta \in \N_0^r \\|\beta|\le 2 }}  f_{i,\beta}(\bar X)  
Y_{1}^iZ_1^{m-i}\bar Y_2^{\beta}Z_2^{2 - |\beta|} \in \R[\bar{X}, Y_{1}, Z_1, \bar Y_2, Z_2]
$$ 
its bihomogeneization only with respect to the variables $Y_{1}$ and $\bar Y_2$, separately.
The additional \linebreak assumption playing the role of condition $(\dag)$ in this case is the following one. 
We say that $f$ satisfies the condition $(\ddagger)$ on $S$ if 
for every $\bar x \in S$, 
$$
\begin{array}{rl}
{\rm i}) & 
\displaystyle{\sum_{\substack{\beta \in \N_0^r \\|\beta|\le 2 }}}  f_{m,\beta}(\bar x)  \bar Y_2^{\beta} 
\hbox{ is positive in } \R^r, \\[10mm]
{\rm ii)} & \displaystyle{\sum_{\substack{\beta \in \N_0^r \\|\beta| =  2 }} } f_{m,\beta}(\bar x) \bar Y_2^{\beta} 
\hbox{ is a positive definite quadratic form in } \R^{r}, \\[10mm]
{\rm iii}) & \hbox{for every } y_1 \in \R, 
\displaystyle{\sum_{0 \le i \le m} \sum_{\substack{\beta \in \N_0^r \\|\beta| = 2 }} } f_{i,\beta}(\bar x) y_{1}^i \bar Y_2^{\beta} 
\hbox{ is a positive definite quadratic form in } \R^{r}. \\[10mm]
\end{array}
$$
The definition of $(\ddagger)$ is made in such a way that if $f$ is positive in $S \times \R^{r+1}$ and satisfies condition $(\ddagger)$ on $S$, 
then $\bar{\bar f}$ is positive on $S \times C \times C^r$.

The extension of Putinar's Positivstellensatz in this final case is the following result.

\begin{theorem}\label{th:2_juegos_de_variables}

Let
$g_1, \dots, g_s \in \R[\bar X]$ 
such that 
$$
\emptyset \ne S  = \{\bar x \in \R^{n} \ | \ g_1(\bar x) \ge 0, \dots, g_s(\bar x) \ge 0\} \subset (-1, 1)^n
$$
and the quadratic module $M(g_1, \dots, g_s) \subset \R[\bar X]$ is archimedean.
There exists a positive constant $c$ such that for every $f \in \R[\bar X, Y_{1}, \bar Y_2]$ positive on $S \times \R^{r+1}$, 
if $\deg_{\bar X} f= d$, $\deg_{Y_{1}} f = m$, 
$\deg_{\bar Y}f=2$,
$f$ satisfies condition $(\ddagger)$ on $S$ 
and
$$
\min \left\{
\bar{\bar f}(\bar x, y_{1}, z_1, \bar y_2, z_2) \, | \, \bar x \in S, (y_{1}, z_1) \in 
C, (\bar y_2, z_2) \in C^r
\right\} = f^{\bullet} > 0,
$$
then $f$ can be written as
$$
f= \sigma_0+\sigma_1g_1+\dots+\sigma_sg_s \in M_{\R[\bar X, Y_{1}, \bar Y_2]}(g_1, \dots, g_s)
$$
with $\sigma_0,\sigma_1, \dots, \sigma_s \in \sum \R[\bar X, Y_{1}, \bar Y_2]^2$ and
$$
\deg(\sigma_0),\deg(\sigma_1g_1), \dots,\deg( \sigma_sg_s) \leq 
c(m+1)2^{\frac{m}2}r^2
{\rm e}^{ \left( \frac{ \| f \|_{\bullet}(m+1)r^2d^2(3n)^d}{{f^{\bullet}}} \right)^{c} }.
$$

\end{theorem}

\section{Proof of the main results}

The proofs of Theorems  \ref{th:2_variables_grado_4}, \ref{th:r_variables_grado_2} and \ref{th:2_juegos_de_variables} follow the same path than the
proof of Theorem \ref{th:main_paper_anterior}
(\cite[Theorem 7]{EscPer}), which is itself mainly a combination and reorganization of 
techniques from 
\cite{NieSchw}, \cite{Pow} and \cite{Schw}.

For $n \in \N$, we denote by $\widetilde{\Delta}_n$ the simplex 
$$
\widetilde{\Delta}_n= \Big\{ \bar x \in \R^n \ | \  \displaystyle \sum_{1 \le i \le n} x_i \leq 1 \hbox{ and } x_i \geq 0 
\hbox{ for } 1 \le i \le n \Big\}
$$
and 
by ${\Delta}_n$ the standard simplex 
$$
{\Delta}_n= \Big\{ (x_0, \bar x) \in \R^{n+1} \ | \  \displaystyle \sum_{0 \le i \le n} x_i = 1 \hbox{ and } x_i \geq 0 
\hbox{ for } 0 \le i \le n \Big\}.
$$


The following two lemmas are slight variations of 
\cite[Lemma 16]{EscPer} and 
\cite[Lemma 11]{NieSchw} (see also \cite[Lemma 17]{EscPer}).

\begin{lemma}\label{lem:norma_div_min}
Let $f \in \R[\bar{X},\bar Y]$ such that $\deg_{\bar{X}}f=d$ and 
$\deg_{\bar Y} f=m$. 
For every $\bar x \in \widetilde \Delta_n$ and $(\bar y, z) \in C^r$,
$$
|\bar f(\bar x, \bar y, z)| \le \| f\|_{\bullet} \binom{m+r}{r}(d+1).
$$
\end{lemma}

\begin{lemma}\label{lema_desig}
Let $f \in \R[\bar{X},\bar Y]$ such that $\deg_{\bar{X}}f=d$ and 
$\deg_{\bar Y} f=m$. For every 
$\bar{x}_1,\bar{x}_2 \in \widetilde{\Delta}_n$ and $(\bar y,z) \in C^r$, 
 $$
 |\bar f(\bar{x}_1,\bar y,z)-\bar f(\bar{x}_2,\bar y,z)| 
 \leq \frac{1}{2} \sqrt{n} \| f \|_{\bullet}\binom{m+r}{r}d(d+1) \| \bar{x}_1-\bar{x}_2\|.
 $$
\end{lemma}

We focus on the proof of Theorem \ref{th:2_variables_grado_4}.
Under the 
stronger assumption of $S 
\subset {\widetilde \Delta}_n^{\circ}$, 
Proposition \ref{prop:main} below is just a slight variant of Theorem \ref{th:2_variables_grado_4} with a 
substantially better degree bound (which unfortunately does not have good rescaling properties). 
Once Proposition \ref{prop:main} is proved, 
Theorem \ref{th:2_variables_grado_4} simply follows by composing with a linear 
change of variables.

\begin{proposition}\label{prop:main}
Let $g_1, \dots, g_s \in \R[\bar{X}]$ such that 
$$
\emptyset \ne S = \{\bar x \in \R^{n} \ | \ g_1(\bar x) \ge 0, \dots, g_s(\bar x) \ge 0\} \subset {\widetilde{\Delta}_n}^{\circ}
$$
and the quadratic module $M(g_1, \dots, g_s)$ is archimedean. 
There exists a positive constant $c$ such that for every $f \in \R[\bar X, Y_1, Y_2]$ positive on $S \times \R^2$, 
if $\deg_{\bar X} f = d$, $\deg_{(Y_1, Y_2)} f = 4$, 
$f$ satisfies condition $(\dag)$ on $S$ and 
$$
\min\{\bar f(\bar x, y_1, y_2, z) \,  | \, \bar x \in S, (y_1, y_2, z) \in C^2 \} = f^{\bullet} >0, 
$$ 
then $f$ can be written as
$$
f = \sigma_0 + \sigma_1g_1 + \dots + \sigma_sg_s \in M_{\R[\bar X, Y_1, Y_2]}(g_1, \dots, g_s)
$$
with $\sigma_0, \sigma_1,  \dots, \sigma_s \in \sum \R[\bar X, Y_1, Y_2]^2$ and 
$$
\deg(\sigma_0), \deg(\sigma_1g_1),  \dots, \deg(\sigma_sg_s)
\le 
c {\rm e}^{\left( \frac{\|  f \|_{\bullet}d^2}{f^{\bullet}} \right)^{c}}.
$$
\end{proposition}

\begin{proof}{Proof:}
Without loss of generality we suppose $\deg g_i \ge 1$ and  $|g_i|\leq 1$ in $\widetilde \Delta_n$ for $1 \le i \le s$. 

If $d = 0$ then $f \in \R[Y_1, Y_2]$ is positive on $\R^2$, and since $\deg_{(Y_1, Y_2)} f = 4$,  $f \in \sum \R[Y_1, Y_2]^2$ and any constant  
$c \ge 4$ works.
So from now we suppose $d \ge 1$ and in the case that the final constant $c$ we find turns out to be less than $4$, we just
replace it by the result of applying \cite[Lemma 18]{EscPer} to 
the $6$-uple $(4, 0, c, 0, 1, c)$. 

We prove first that there exist 
$\lambda \in \R_{>0}$ and $k \in \N_0$ 
such that
$$
h= \bar{f}- \lambda\big( Y_1^2 + Y_2^2 + Z^2 \big)^2 
\sum_{1 \leq i \leq s} g_i \cdot (g_i-1)^{2k}  \in \R[\bar{X},Y_1, Y_2,Z]
$$
satisfies $h \ge \frac12 f^{\bullet}$ in $\widetilde{\Delta}_n \times 
C^2$.

For each $(y_1, y_2,z) \in C^2$ we consider 
$$
A_{y_1, y_2,z}=\left\{\bar{x} \in \widetilde{\Delta}_n \ | \ \bar{f}(\bar x, y_1, y_2, z) \leq \frac{3}{4}f^{\bullet}\right\}.
$$
Note that $A_{y_1, y_2, z} \cap S = \emptyset$. 
To exhibit sufficient conditions for $\lambda$ and $k$, we consider separately the cases 
$\bar{x} \in \widetilde{\Delta}_n \setminus A_{y_1, y_2,z}$ and 
$\bar{x} \in A_{y_1, y_2,z}$.

If $\bar{x} \in \widetilde{\Delta}_n \setminus A_{y_1, y_2,z}$ 
\begin{equation*}
\begin{split}
 h(\bar{x},y_1, y_2,z) & =\bar{f}(\bar{x},y_1, y_2,z)- 
 \lambda \big( y_1 ^2 + y_2^{2} + z^{2} \big)^{2} 
  \sum_{1 \leq i \leq s} g_i(\bar{x}) \cdot (g_i(\bar{x})-1)^{2k} \\
  & \ge \bar{f}(\bar{x},y_1, y_2,z)- 
 \lambda  
  \sum_{1 \leq i \leq s} |g_i(\bar{x})| \cdot (|g_i(\bar{x})|-1)^{2k} \\
 &> \frac{3}{4}f^{\bullet} - 
  \frac{\lambda s}{2k+1} 
 \end{split}
\end{equation*}
using \cite[Remark 12]{NieSchw}.
Therefore the condition $h(\bar{x},y_1, y_2,z) \ge \frac 12 f^{\bullet}$ is ensured if
\begin{equation}\label{eq:cond_k}
2k+1\ge \frac{4 \lambda  s}{f^{\bullet}}.
\end{equation}

If $\bar{x} \in A_{y_1, y_2,z}$, for any $\bar{x}_0 \in S$, by Lemma \ref{lema_desig} with $r = 2$ and $m = 4$,  we have
$$
\frac{f^{\bullet}}{4} \leq \bar{f}(\bar{x}_0, y_1, y_2, z)-\bar{f}(\bar{x}, y_1, y_2, z) \leq \frac{15}{2} \sqrt{n} \| f \|_{\bullet}
d(d+1) \| \bar{x}_0-\bar{x}\|,
$$
then
$$
\frac{f^{\bullet}}{30\sqrt{n} \| f \|_{\bullet}d(d+1)} 
\leq \| \bar{x}_0-\bar{x}\| 
$$
and therefore
\begin{equation}\label{eq:dist}
\frac{f^{\bullet}}{30\sqrt{n} \| f \|_{\bullet}d(d+1)} \leq \mbox{dist}(\bar{x},S).
\end{equation}

Using \cite[Remark 13]{EscPer}, there exist $c_1, c_2 > 0$ and 
$1 \le i_0 \le s$ such that $g_{i_0}(\bar x) < 0$ and
\begin{equation}\label{eq:loja}
\mbox{dist}(\bar x, S)^{c_1} \le -c_2g_{i_0}(\bar x).
\end{equation}
By (\ref{eq:dist}) and (\ref{eq:loja}) we have
\begin{equation}\label{desig_6}
 g_{i_0}(\bar{x})\leq - \delta.
\end{equation}
with 
$$
\delta= \frac{1}{c_2} \left( \frac{f^{\bullet}}{30\sqrt{n}\| f \|_{\bullet}d(d+1)} \right)^{c_1} > 0.
$$

On the other hand, defining $f^{\bullet}_{y_1, y_2,z}= 
\min\{\bar f(\bar x, y_1, y_2, z) \,  | \, \bar x \in S \}$,
again by Lemma \ref{lema_desig} with $r = 2$ and $m = 4$, we have that
\begin{equation}\label{desig_7}
|\bar{f}(\bar{x},y_1, y_2,z)-f^{\bullet}_{y_1, y_2,z}| \leq \frac{15}{2} \sqrt{n}\| f \|_{\bullet}d(d+1) 
\mbox{diam} (\widetilde{\Delta}_n)
= \frac{15}{\sqrt{2}} \sqrt{n}\| f \|_{\bullet}d(d+1).
\end{equation}

Then, using again \cite[Remark 12]{NieSchw}, (\ref{desig_6}) and (\ref{desig_7}) we have 
\begin{equation*}
 \begin{split}
 h(\bar{x},y_1, y_2,z)& \geq \bar{f}(\bar{x},y_1, y_2,z)
 -\lambda g_{i_0}(\bar{x})(g_{i_0}(\bar{x})-1)^{2k} - 
  \frac{\lambda(s-1)}{2k+1} \\
 & \geq \bar{f}(\bar{x},y_1, y_2,z) -f^{\bullet}_{y_1, y_2,z} +f^{\bullet}_{y_1, y_2,z} + 
 \lambda  \delta - \frac{\lambda(s-1)}{2k+1}  \\
 & \geq - \frac{15}{\sqrt{2}} \sqrt{n}\|f \|_{\bullet}d(d+1) 
 + f^{\bullet} + \lambda \delta - \frac{\lambda(s-1)}{2k+1}.
 \end{split}
\end{equation*}
Finally, the condition $h(\bar{x},y_1, y_2,z) \geq \frac12f^{\bullet}$ is ensured if
\begin{equation}\label{eq:cond_lambda}
\lambda \ge \frac{ 15\sqrt{n}\|f \|_{\bullet}d(d+1)}{\sqrt{2}\delta}  = 
\frac{ c_2 2^{c_1} (15\sqrt{n}\| f \|_{\bullet}d(d+1))^{c_1+1}}{\sqrt{2}{f^{\bullet}}^{c_1}}
\end{equation}
and 
\begin{equation}\label{eq:aux_inutil}
2k+1 \ge \frac{2\lambda(s-1)}{f^{\bullet}}. 
\end{equation}
Since (\ref{eq:cond_k}) implies (\ref{eq:aux_inutil}), it is enough for $\lambda$ and $k$ to satisfy (\ref{eq:cond_k}) and (\ref{eq:cond_lambda}). 
So for the rest of the proof we take
$$
\lambda = \frac{ c_2 2^{c_1} (15\sqrt{n}\| f \|_{\bullet}d(d+1))^{c_1+1}}{\sqrt{2}{f^{\bullet}}^{c_1}} 
= c_3\frac{(15\| f \|_{\bullet}d(d+1))^{c_1+1}}{{f^{\bullet}}^{c_1}}
> 0$$ 
with $c_3 = \frac{c_22^{c_1}\sqrt{n}^{c_1+1}}{\sqrt{2}}$
and
$$
k = \left\lceil \frac12 \left( \frac{4\lambda s}{f^{\bullet}} -1\right) \right\rceil \in \N_0.
$$
In this way, 
\begin{equation}\label{eq:cota_k}
\begin{split}
 k & \le   \frac12 \left( \frac{4\lambda s}{f^{\bullet}} -1\right) +1  \\
   & = 2c_3s \left( \frac{ 15\| f \|_{\bullet}d(d+1)}{f^{\bullet}} \right)^{c_1+1}+\frac12 \\
   & \leq c_4 \left( \frac{15\| f \|_{\bullet}d(d+1)}{f^{\bullet}} \right)^{c_1+1}
 \end{split}
\end{equation}
with $c_4 = 2c_3s + 1$. Here (and also several times after here) we use Lemma \ref{lem:norma_div_min} with $r = 2$ and $m = 4$ to ensure
$$\frac{15\| f \|_{\bullet}(d+1)}{f^{\bullet}} \ge 1.$$

Also, if we define $\ell = \deg_{\bar X} h$, we have
\begin{equation}\label{cota_e}
 \begin{split}
\ell & \leq \max \{d, (2k+1)\max_{1 \le i \le s} \deg g_i  \} \\ 
& \leq \max \left\{
d,  
\left(2c_4\left( \frac{15\| f \|_{\bullet}d(d+1)}{f^{\bullet}} \right)^{c_1+1} +1 \right) 
\max_{1 \le i \le s} \deg g_i
\right\} \\
 &
\le
c_5 \left( \frac{15\| f\|_{\bullet}d(d+1)}{f^{\bullet}} \right)^{c_1+1}
 \end{split}
\end{equation}
with $c_5 = (2c_4 + 1)\displaystyle{\max_{1 \le i \le s} \deg g_i}$.

On the other hand, using conveniently \cite[Proposition 14]{NieSchw} and (\ref{eq:cota_k}), 
\begin{equation}\label{cota_h}
 \begin{split}
\|h\|_{\bullet} 
&
\le  {\| f \|_{\bullet}} + 2\lambda s     
\max_{1 \le i \le s} \{(\deg g_i +1) (\|g_i\| +1)\} ^{2k+1}   \\[3mm]
& =  
\| f \|_{\bullet} +  
2c_3s\frac{ (15\| f \|_{\bullet}d(d+1))^{c_1+1}}{{f^{\bullet}}^{c_1}}
\max_{1 \le i \le s} \{(\deg g_i +1) (\|g_i\| +1)\}^{2k+1} \\[3mm]
& \le 
(2c_3s+1)\max_{1 \le i \le s} \{(\deg g_i +1) (\|g_i\| +1)\} \ \cdot 
\\
& \ \ \ \cdot \
\frac{ (15\| f \|_{\bullet}d(d+1))^{c_1+1}}{{f^{\bullet}}^{c_1}}
\max_{1 \le i \le s} \{(\deg g_i +1) (\|g_i\| +1)\}^{2c_4 \left( \frac{15\| f \|_{\bullet}d(d+1)}{f^{\bullet}} \right)^{c_1+1}}
\\[3mm]
& =
c_6 
\frac{ (15\| f \|_{\bullet}d(d+1))^{c_1+1}}{{f^{\bullet}}^{c_1}}
{\rm e}^{c_7 \left( \frac{15\| f \|_{\bullet}d(d+1)}{f^{\bullet}} \right)^{c_1+1} }
\end{split}
\end{equation}
with 
$c_6 = (2c_3s + 1)\displaystyle{\max_{1 \le i \le s} \{(\deg g_i +1) (\|g_i\| +1)\}}$ 
and 
$c_7 = \log \left(\displaystyle{\max_{1 \le i \le s} \{(\deg g_i +1) (\|g_i\| +1)\}^{2c_4}}\right)$.

So far we have found $\lambda$ and $k$ such that 
that $h \ge \frac12 f^{\bullet}$ in $\widetilde \Delta_n \times C^2$, together with bounds for 
$k, \ell = \deg_{\bar X} h$ and $\|h\|_\bullet$. 
Now, 
we introduce a new variable $X_0$ to  
homogenize with respect to the variables $\bar X$ and  
use P\'olya's Theorem.
Let 
$$
h=  \sum_{\substack{\beta \in \N_0^2 \\ |\beta|\leq 4}}  \sum_{0 \leq j \leq \ell} h_{j, \beta}(\bar{X})Y_1^{\beta_1}Y_2^{\beta_2}Z^{4-|\beta|} 
$$
with $h_{j, \beta} \in \R[\bar X]$ equal to zero or homogeneous of degree  $j$ for $\beta \in \N_0^2$, $0 \le |\beta| \le  4$ and 
$0 \le j \le \ell$. 
We define
$$
H= \sum_{\substack{\beta \in \N_0^2 \\ |\beta|\leq 4}}  \sum_{0 \leq j \leq \ell} h_{j, \beta}(\bar{X})
(X_0+X_1 \cdots +X_n)^{\ell-j} Y_1^{\beta_1}Y_2^{\beta_2} Z^{4-|\beta|} \in \R[X_0, \bar{X},Y_1, Y_2,Z]
$$
which is bihomogeneous in $(X_0, \bar X)$ and $(Y_1, Y_2, Z)$ of bidegree $(\ell, 4)$. 

Since $H(x_0, \bar x, y_1, y_2, z) = h(\bar x, y_1, y_2, z)$ for every $(x_0, \bar x, y_1, y_2, z) \in 
\Delta_n \times C^2$, it is clear that $H\ge\frac12f^{\bullet}$ in $\Delta_n \times C^2$.

On the other hand, for each $(y_1, y_2, z) \in C^2$, we consider $H(X_0, X, y_1, y_2, z) \in 
\R[X_0, \bar X]$. Using again \cite[Proposition 14]{NieSchw} we have 
$$
\begin{array}{rcl}
\|H(X_0, X,y_1, y_2,z) \| & 
\leq  & \displaystyle{\sum_{\substack{\beta \in \N_0^2 \\ |\beta|\leq 4}}   \sum_{0 \leq j \leq \ell} \|h_{j, \beta}(\bar X)(X_0+\cdots+X_n)^{\ell-j}y_1^{\beta_1}y_2^{\beta_2}z^{4-|\beta|} \|   }
\\[6mm]
& \leq & \displaystyle{ \sum_{\substack{\beta \in \N_0^2 \\ |\beta|\leq 4}}  \sum_{0 \leq j \leq \ell} \|h_{j, \beta }(\bar X)(X_0+\cdots+X_n)^{\ell-j} \|   }
\\[6mm]
& \leq & \displaystyle{ \sum_{\substack{\beta \in \N_0^2 \\ |\beta|\leq 4}}  \sum_{0 \leq j \leq \ell} \|h_{j, \beta}(\bar X) \| }  
\\[6mm] 
& \leq  & 15(\ell + 1)\| h \|_\bullet.
\end{array}
$$

We use now the bound for P\'olya's Theorem from \cite[Theorem 1] {PR}.
Take $N \in \N$ given by 
$$
N = \left\lfloor \frac{15(\ell+1)\ell(\ell-1)\|h\|_{\bullet}}{f^{\bullet}} - \ell \right\rfloor + 1.
$$
Then for each $(y_1, y_2 ,z) \in C^2$ we have that 
$
H\big(X_0,{\bar{X}},y_1, y_2,z\big)\left(X_0 + X_1 + \cdots + X_n \right)^N \in \R[X_0, \bar X]$
is a homogeneous polynomial such that all its coefficients are positive. 
More precisely, suppose we write
\begin{equation}\label{desig_8}
H\big(X_0,{\bar{X}},Y, Z\big)\left(X_0 + X_1 + \cdots + X_n \right)^N
= 
\sum_{\substack{\alpha=(\alpha_0,\bar \alpha) \in \N_0^{n+1} \\ |\alpha|=N+\ell}} 
b_{\alpha}(Y_1, Y_2, Z) X_0^{\alpha_0} {\bar{ X}}^{\bar \alpha}
\in \R[X_0, \bar X, Y, Z]
\end{equation}
with $b_\alpha \in \R[Y_1, Y_2, Z]$ homogeneous of degree $4$.
The conclusion is that for every $\alpha \in \N_0^{n+1}$ with $|\alpha|=N+\ell$, 
the polynomial $b_\alpha$ is positive in $C^2$, and therefore, 
since it is a homogenous polynomial, $b_\alpha$
is non-negative in $\R^3$. 

Before going on, we bound $N + \ell$ using (\ref{cota_e}) and (\ref{cota_h}) as follows.

\begin{equation}\label{eq:cota_Nmasell}
 \begin{split}
  N+ \ell & \leq \frac{15(\ell +1)\ell (\ell-1)\|h\|_{\bullet}}{f^{\bullet}}+1 \\
 & \leq \frac{15\ell^3\|h\|_{\bullet}}{f^{\bullet}}+1  \\
 & \leq 
15c_5^3c_6
\left(\frac{ 15\| f \|_{\bullet}d(d+1)}{{f^{\bullet}}}\right)^{4(c_1+1)}
{\rm e}^{c_7 \left( \frac{15\| f \|_{\bullet}d(d+1)}{f^{\bullet}} \right)^{c_1+1} }
  +1  \\
 & \leq 
c_8
\left(\frac{ 15\| f \|_{\bullet}d(d+1)}{{f^{\bullet}}}\right)^{4(c_1+1)}
{\rm e}^{c_7 \left( \frac{15\| f \|_{\bullet}d(d+1)}{f^{\bullet}} \right)^{c_1+1} }
  \\
 \end{split}
\end{equation}
with $c_8 = 15c_5^3c_6 + 1$.

Now we substitute $X_0=1-X_1- \cdots -X_n$ and $Z =1$ in (\ref{desig_8})
and we obtain
\begin{equation} \label{eq:almost_finish}
f =  
\lambda\big(Y_1^2 + Y_2^2 +  1\big)^{2}\sum_{1 \le i \le s}g_i(g_i-1)^{2k} 
+
\sum_{\substack{\alpha=(\alpha_0,\bar \alpha) \in \N_0^{n+1} \\ |\alpha|=N+\ell}}  
b_{\alpha}(Y_1, Y_2, 1) (1 - X_1 - \dots - X_n)^{\alpha_0} {\bar{ X}}^{\bar \alpha} \in \R[\bar X, Y_1, Y_2].
\end{equation}
From (\ref{eq:almost_finish}) we want to conclude that $f \in M_{\R[\bar X, Y_1, Y_2]}(g_1 \dots, g_s)$
and to find the positive constant $c$ such that the degree bound holds.

The first term on the right hand side 
of (\ref{eq:almost_finish}) clearly
belongs to $M_{\R[\bar X, Y_1, Y_2]}(g_1 \dots, g_s)$. Moreover, 
for $1 \le i \le s$, 
\begin{equation}\label{eq:cota_1er_term}
\deg \big(Y_1^2 + Y_2^2 + 1\big)^{2}g_i(g_i-1)^{2k} = 4 + (2k+1)\deg g_i.
\end{equation}

Now we focus on the second term on the right hand side 
of (\ref{eq:almost_finish}), which is itself a sum. Take a fixed 
$\alpha \in \N_0^{n+1}$ with $|\alpha|=N+\ell$.

Since 
$b_\alpha(Y_1, Y_2, 1)$ is non-negative in $\R^2$ and $\deg_{(Y_1, Y_2)}b_\alpha(Y_1, Y_2, 1) \le 4$, $b_\alpha(Y_1, Y_2, 1) \in \sum \R[Y_1, Y_2]^2$. 
Moreover, we can write $b_\alpha(Y_1, Y_2, 1)$ as a sum of squares with 
the degree of each square bounded by $4$.

Also, take $v(\alpha) =  (v_0, \bar v) \in \{0, 1\}^{n+1}$ 
such that $\alpha_i \equiv v_i \, (\mbox{mod} \ 2)$ for $0 \le i \le n$. 
Denoting $g_0 = 1 \in \R[\bar X]$, since $S \subset {\widetilde \Delta_n}^\circ$, by classical Putinar's Positivstellensatz we have representations
$$
(1-X_1- \cdots -X_n)^{v_0} \bar{ X}^{\bar v} =  \sum_{0 \le i \le s}
\sigma_{v(\alpha) i} g_i, 
$$
with $\sigma_{v(\alpha)i} \in \sum \R[\bar X]^2$ for $0 \le i \le s$, and 
then 
$$
(1 - X_1 - \dots - X_n)^{\alpha_0} {\bar{ X}}^{\bar \alpha}
= 
(1 - X_1 - \dots - X_n)^{\alpha_0-v_0} {\bar{ X}}^{\bar \alpha - \bar v}
\sum_{0 \le i \le s}
\sigma_{v(\alpha) i} g_i
$$
belongs to $M(g_1, \dots, g_s)$
since $(1 - X_1 - \dots - X_n)^{\alpha_0-v_0} {\bar{ X}}^{\bar \alpha - \bar v}
\in \R[\bar X]^2$. 

We conclude that each term in the sum belongs to $M_{\R[\bar X, Y_1, Y_2]}(g_1, \dots, g_s)$.
In addition, 
for $0 \le i \le s$ we have
\begin{equation}\label{eq:cota_2do_term}
\deg 
b_{\alpha}(Y_1, Y_2, 1) (1 - X_1 - \dots - X_n)^{\alpha_0-v_0} 
{\bar{ X}}^{\bar \alpha - \bar v} \sigma_{v(\alpha)i}g_i 
\le 4 + N + \ell + c_9
\end{equation}
with $c_9 = \max \{ \deg \sigma_{vi}g_i \ |
\ v \in \{0, 1\}^{n+1},  0 \le i \le s\}$.

To finish the proof, we only need to bound simultaneously the right hand side
of (\ref{eq:cota_1er_term}) and 
(\ref{eq:cota_2do_term}).

On the one hand, using (\ref{eq:cota_k}), 
\begin{equation*}
\begin{split}
4 + (2k+1)\max_{1 \le i \le s} \deg g_i & \le 
4 + \left(2 c_4 \left( \frac{15\| f \|_{\bullet}d(d+1)}{f^{\bullet}} \right)^{c_1+1} + 1\right) 
\max_{1 \le i \le s} \deg g_i \\
& \le
c_{10}  \left(\frac{\| f \|_{\bullet}d^2}{f^{\bullet}} \right)^{c_1+1}  
\end{split}
\end{equation*}
with $c_{10} = (2c_4 + 5)30^{c_1 + 1}\displaystyle{\max_{1 \le i \le s} \deg g_i}$, since 
$d \ge 1$. 

On the other hand, using (\ref{eq:cota_Nmasell}), 
\begin{equation*}
\begin{split}
4 + N + \ell + c_9 &
\le 
4 + 
c_8
\left(\frac{ 15\| f \|_{\bullet}d(d+1)}{{f^{\bullet}}}\right)^{4(c_1+1)}
{\rm e}^{c_7 \left( \frac{15\|\bar f \|_{\bullet}d(d+1)}{f^{\bullet}} \right)^{c_1+1} } + c_9 \\
&
\le
c_{11}
\left(\frac{ \| f \|_{\bullet}d^2}{{f^{\bullet}}}\right)^{4(c_1+1)}
{\rm e}^{c_{12} \left(  \frac{\|f \|_{\bullet}d^2}{f^{\bullet}} \right)^{c_1+1} }  \\
\end{split}
\end{equation*}
with $c_{11} = (4 + c_8 + c_9)30^{4(c_1+1)}$ and $c_{12} = c_7 30^{c_1 + 1}$, again since 
$d \ge 1$. 

Finally, we define $c$ as the positive constant obtained applying \cite[Lemma 18]{EscPer}  to 
the $6$-uple $(c_{10},$ $c_1+1,$ $c_{11}, 4(c_1 + 1), c_{12}, c_1 + 1)$.
\end{proof}

We are ready now to prove Theorem \ref{th:2_variables_grado_4}. 
The proof consists basically 
in a linear change of variables and the application of 
Proposition \ref{prop:main}, as in the proof of \cite[Theorem 7]{EscPer}.

\begin{proof}{Proof of Theorem \ref{th:2_variables_grado_4}:}

We consider the affine change of variables $\ell:\R^n \to \R^n$ given by
$$
\ell(X_1, \dots, X_n) = \left(\frac{X_1 + 1}{2n}, \dots, 
\frac{X_n + 1}{2n}\right).
$$
For $0 \le i \le s$, we take 
$\tilde{g}_i(\bar X)=g_i(\ell^{-1}(\bar X)) \in \R[\bar X]$ and we define
$$
\widetilde{S}= \{\bar x \in \R^n \ | \ \tilde{g}_1(\bar x) \geq 0, \dots,  
\tilde{g}_s(\bar x) \ge 0 \}.
$$
It is easy to see that
$$
\emptyset \ne \widetilde{S} = \ell(S) \subseteq \widetilde{\Delta}_n^{\circ}.
$$
Moreover, since $M(g_1, \dots, g_n)$ is archimedean, 
$M(\tilde g_1, \dots, \tilde g_s)$ is also archimedean
(see \cite[Proof of Theorem 7]{EscPer}).

Let 
$f \in \R[\bar X, Y_1, Y_2]$ be as in the statement of Theorem \ref{th:2_variables_grado_4}
and let $\tilde{f}(\bar{X},Y_1, Y_2)=f(\ell^{-1}(\bar{X}),Y_1, Y_2) \in \R[\bar X, Y_1, Y_2]$.
It can be easily seen that
$\tilde f$ is positive on $\tilde S \times \R^2$, 
$\deg_{\bar X}\tilde f = \deg_{\bar X} f = d$, 
$\deg_{(Y_1, Y_2)}\tilde f = \deg_{(Y_1, Y_2)} f = 4$, $\tilde f$ satisfies condition 
$(\dag)$ on $\widetilde S$
and
$$
\min\{\bar{\tilde f}(\bar x, y_1, y_2, z) \,  | \, \bar x \in \widetilde S, (y_1, y_2, z) \in C^2 \} 
= \min\{\bar{f}(\bar x, y_1, y_2, z) \,  | \, \bar x \in  S, (y_1, y_2, z) \in C^2 \} = f^{\bullet} >0. 
$$ 
In addition, $\| \tilde f \|_{\bullet} \le \| f \|_{\bullet}(3n)^d$
(again, see \cite[Proof of Theorem 7]{EscPer}).

Take $c$ as the positive constant from Proposition \ref{prop:main} applied to 
$\tilde g_1, \dots, \tilde g_s$. Therefore, 
$\tilde f$ can be written as 
$$
\tilde f = \tilde \sigma_0 + \tilde \sigma_1\tilde g_1 + \dots + \tilde\sigma_s \tilde g_s \in 
M_{\R[\bar X, Y]}(\tilde g_1, \dots, \tilde  g_s)
$$
with $\tilde\sigma_0, \tilde\sigma_1,  \dots, \tilde\sigma_s \in \sum \R[\bar X, Y_1, Y_2]^2$ and 
$$
\deg(\tilde\sigma_0), \deg(\tilde\sigma_1 \tilde g_1),  \dots, \deg(\tilde\sigma_s\tilde g_s)
\le 
c {\rm e}^{\left( \frac{\|  f \|_{\bullet}d^2(3n)^d}{f^{\bullet}} \right)^{c}}
$$
and the final representation for $f$ is simply obtained by composing with $\ell$. 
\end{proof}

The proof of Theorem \ref{th:r_variables_grado_2} is 
a straightforward adaptation of the proof of 
Theorem \ref{th:2_variables_grado_4} (and Proposition \ref{prop:main}), and we omit it.

To prove Theorem \ref{th:2_juegos_de_variables}, we need the following 
two auxilary lemmas, which are again slight variations of 
\cite[Lemma 16]{EscPer} and 
\cite[Lemma 11]{NieSchw}.

\begin{lemma}
Let $f \in \R[\bar{X},Y_{1}, \bar Y_2]$ such that $\deg_{\bar{X}}f=d$ 
$\deg_{Y_{1}}f=m$, 
and 
$\deg_{\bar Y_2} f=2$. 
For every $\bar x \in \widetilde \Delta_n$, $(y_1, z_1) \in C$ 
and $(\bar y_2, z_2) \in C^r$,
$$
|\bar{\bar f}(\bar x, y_{1}, z_1, \bar y_2, z_2)| \le \frac12\| f\|_{\bullet} (m+1)(r+1)(r+2)(d+1).
$$
\end{lemma}

\begin{lemma}
Let $f \in \R[\bar{X},Y_{1}, \bar Y_2]$ such that $\deg_{\bar{X}}f=d$ 
$\deg_{Y_{1}}f=m$, 
and 
$\deg_{\bar Y_2} f=2$. 
For every 
$\bar{x}_1,\bar{x}_2 \in \widetilde{\Delta}_n$ 
$(y_1, z_1) \in C$ 
and $(\bar y_2, z_2) \in C^r$, 
$$
 |\bar{\bar f}(\bar{x}_1,y_{1}, z_1, \bar y_2, z_2)-\bar{\bar f}(\bar{x}_2,y_{1}, z_1, \bar y_2, z_2)| 
 \leq \frac{1}{4} \sqrt{n} \| f \|_{\bullet}(m+1)(r+1)(r+2)d(d+1) \| \bar{x}_1-\bar{x}_2\|.
 $$
\end{lemma}

Then, the proof of Theorem \ref{th:2_juegos_de_variables} is also 
a straightforward adaptation of the proof of Theorem 
\ref{th:2_variables_grado_4} (and Proposition \ref{prop:main}), with the only caveat that the auxiliary polynomial 
$h$ (at the beginning of the proof of Proposition \ref{prop:main}) should be defined as
$$
h= \bar{\bar{f}}- \lambda\big( Y_1^2 + Z_1^2)^{\frac{m}2}(Y_{21}^2 + \dots + Y_{2r}^2 +  Z_2^2 \big) 
\sum_{1 \leq i \leq s} g_i \cdot (g_i-1)^{2k}  \in \R[\bar{X},Y_1, Z_1,  \bar Y_2,Z_2]
$$
and then
$$
\| h \|_\bullet \le \| f \|_\bullet + \lambda s 2^{\frac{m}2}\max_{1\le i \le s}\{(\deg g_i + 1)(\|g_i\| + 1)  \}^{2k+1}.
$$

\bigskip

\textbf{Acknowledgements:} We are very grateful to Michel Coste for 
suggesting us these extensions of our results from \cite{EscPer}.

\end{document}